\documentclass[leqno]{siamltex}
\usepackage{amsmath,amssymb}
\usepackage{color}

\newtheorem{remark}[theorem]{{\sc Remark}}

\newcommand{\eps}{\varepsilon}

\usepackage{graphicx}

\title{Perturbed asymptotic expansions for interior-layer solutions  of a
semilinear reaction-diffusion problem with small diffusion
\thanks{This research was supported by Science Foundation Ireland grant 04/BR/M0055.
The first author was also supported by Science Foundation Ireland
grant 08/RFP/MTH1536.}}

\author{Natalia Kopteva\thanks{Department of Mathematics and Statistics,
University of Limerick, Limerick, Ireland (natalia.kopteva@ul.ie)}
\and  Martin Stynes\thanks{Department of Mathematics,
National University of Ireland, Cork, Ireland (m.stynes@ucc.ie)} }

\begin{document}

\maketitle

\begin{abstract}
A semilinear reaction-diffusion two-point boundary value problem,
whose second-order derivative is multiplied by a small positive
parameter $\eps^2$, is considered. It can have multiple solutions.
An asymptotic expansion is constructed for a solution that has an
interior layer. Further properties are then established for a
perturbation of this expansion. These are used in~\cite{KoStMain}
to obtain discrete sub-solutions and super-solutions for certain
finite difference methods described there, and in this way yield
convergence results for those methods.
\end{abstract}



\pagestyle{myheadings}
    \thispagestyle{plain}
    \markboth {}
    {}

\section{Introduction}

We are interested in interior-layer solutions of
the singularly perturbed semilinear reaction-diffusion
boundary-value problem
\begin{subequations}\label{1.1}
\begin{eqnarray}
    &F u(x) \equiv -\eps^2 u''(x)+ b(x,u) = 0
        \quad \mbox{for } x\in (0,1),& \label{1.1a}\\
         &u(0)=g_0, \;\; u(1)=g_1,& \label{1.1b}
\end{eqnarray}
\end{subequations}
where $\eps$ is a small positive parameter, $b\in
C^\infty([0,1]\times \mathbb{R})$, and $g_0$ and $g_1$ are given
constants.

Under the hypotheses that are stated below,
this problem can have multiple
solutions that exhibit interior-layer behaviour.
A companion paper~\cite{KoStMain}
discusses numerical methods for its solution. In this present report
we shall present some of the details that are omitted
from~\cite{KoStMain}.

The reduced problem of (1.1) is defined by formally setting
$\eps=0$ in (\ref{1.1a}), viz.,
\begin{equation}\label{red}
b(x,\varphi)=0 \quad \mbox{for } x\in (0,1).
\end{equation}
Assume that  the reduced problem (\ref{red}) has three simple
roots $\varphi=\varphi_k\in C^\infty[0,1]$ for $k=0,1,2$:
$$
b(x,\varphi_k(x))=0\qquad
\text{for}\;\; k=0,1,2\;\;  \text{and}\;\; x\in
    [0,1]
\leqno{(A1)}
$$
where
$$
\leqno{(A2)}
\qquad
\begin{cases}
\varphi_1(x)<\varphi_0(x)<\varphi_2(x)
    &\text{for}\ x\in [0,1]\\
&\hspace{-4cm}\text{ and there is no other solution of \eqref{red} between }
    \varphi_1\text{ and }\varphi_2.
\end{cases}
$$
Here and subsequently, numbering such  as $(A1)$ indicates an
\emph{assumption} that holds true throughout the paper. Assume also
that
$$
b_u(x,\varphi_k(x))>\gamma^2>0
\qquad\text{for}\;\;k=1,2\;\;  \text{and}\;\; x\in
    [0,1]
    \vspace{-0.1cm}
\leqno{(A3)}
$$
but\vspace{-0.1cm}
$$
b_u(x,\varphi_0)<0\qquad \text{for}\;\; x\in
    [0,1]. 
\leqno{(A4)}
$$

Assumption $(A3)$ says that  $\varphi_1(x)$ and $\varphi_2(x)$ are
stable reduced solutions, i.e., one may have a solution $u$ of
$(\ref{1.1})$ that is very close to either $\varphi_1$ or
$\varphi_2$
on some subdomain of $(0,1)$.
Assumption $(A4)$ means that the solution $\varphi_0(x)$ is unstable:
no solution of $(\ref{1.1})$ lies close to $\varphi_0$
on any subdomain of (0,1).
Under the hypotheses $(A1)$--$(A4)$, the equation (\ref{1.1a})
is often described as bistable.

Our further assumption is that the equation
$\int_{\varphi_1(x)}^{\varphi_2(x)} b(x,v)\,dv=0$ has a solution
$x=t_0$ such that
$\frac{d}{dx}\big[\int_{\varphi_1(x)}^{\varphi_2(x)} b(x,v)\,dv
\big]\bigr|_{x=t_0}\neq 0, $ i.e., this root is simple. As in many
asymptotic analysis papers, we also assume that the value of this derivative is
negative, since this sign corresponds to the Lyapunov stability of
an interior-layer solution $u(x)$  of (\ref{1.1}) that switches from
$\varphi_1$ to $\varphi_2$ when $u$ is regarded as a steady-state
solution of the time-dependent parabolic problem $v_t -\eps^2
v_{xx}+ b(x,v) =0$
(see \cite[Section 7, Remark~3]{BVN97};
if instead the derivative were positive, this
would correspond to Lyapunov stability of an interior-layer solution
that switches from $\varphi_2$ to $\varphi_1$).
By Assumption $(A1)$
these hypotheses on the integrals of $b$ are equivalent to the
assumptions
$$
\int_{\varphi_1(t_0)}^{\varphi_2(t_0)}\!\! b(t_0,v)\,dv =0
\qquad \text{and} \qquad
    \int_{\varphi_1(t_0)}^{\varphi_2(t_0)}\!\!
b_x(t_0,v)\,dv =-C_I< 0.  
\leqno{(A5)}
$$
Similar conditions are assumed in \cite[\S 4.15.4]{VB},
\cite[\S2.3.2]{VB2} and also in \cite{FifeGr,Nef} for an analogous
two-dimensional problem and \cite{Fife76} for a analogous system of equations.

\begin{remark}
Assumption $(A2)$ can be relaxed to allow other roots of (\ref{red})
between $\varphi_1$ and $\varphi_2$
provided that
$\int_{\varphi_1(t_0)}^{v}\!\! b(t_0,s)\,ds >0$
for all $v\in(\varphi_1(t_0),\varphi_2(t_0))$.
Note that this  inequality follows immediately from
 $(A1)$--$(A5)$ if $\varphi_0$ is the only reduced solution between $\varphi_1$
 and $\varphi_2$.
\end{remark}

The solutions $\varphi_1$ and $\varphi_2$ of (\ref{red}) do not in
general satisfy either of the boundary conditions in (\ref{1.1b}).
In order to focus on interior layers, we exclude boundary layers
 by assuming that
$$
\varphi_1(0)=g_0,\qquad
\varphi_2(1)=g_1, \qquad
\varphi''_1(0)=\varphi''_2(1)=0.     
\leqno{(A6)}
$$

Under Assumptions $(A1)$--$(A6)$, the
problem
(\ref{1.1}) has a solution that, roughly speaking, lies in the
neighbourhood of $\varphi_1(x)$ and $\varphi_2(x)$ for
$x\in[0,t_0)$ and $x\in(t_0,1]$ respectively
(see \cite[Corollary~6.7]{KoStMain}). Near $x=t_0$ the
solution switches from $\varphi_1$ to $\varphi_2$, which results
in an interior transition layer of width $O(\eps|\ln\eps|)$.

The structure of the paper is as follows. In
Section~\ref{sec:asymexp} an asymptotic expansion of a
interior-layer solution of (\ref{1.1}) is constructed; this analysis
draws on ideas of \cite{Fife,KoSt04,Nef,VB2}. A modified version of
this asymptotic expansion, related to \cite{Nef}, is examined in
Section~\ref{sec_per_as_exp}. This modified expansion is used in \cite{KoStMain}
to construct discrete sub-solutions and super-solutions for the numerical methods
used in that paper to solve~\eqref{1.1}.
\smallskip

\noindent{\it Notation\/}. Throughout the paper, $C, C', \bar C$ and $\bar
C'$, sometimes subscripted, denote generic positive constants that
are independent of $\eps$ and of the mesh; furthermore, $\bar C$
and $\bar C'$ are taken sufficiently large where this property is
needed. These constants may take different values in different
places. Notation such as $f=O(z)$ means $|f| \le Cz$ for some $C$.

\section{Asymptotic expansion for the continuous problem}\label{sec:asymexp}
The point $t_0\in(0,1)$ is fixed by Assumption $(A5)$. Define the
stretched variable $\xi$ by
$$
\xi:=(x-t_0)/\eps. 
$$
Then a standard calculation shows that the zero-order
interior-layer term $V_0(\xi)$ of the asymptotic expansion of $u$
is given by a solution of the following problem:
\begin{subequations}
\begin{equation}\label{V_0_nat}
-{\textstyle\frac{d^2}{d \xi^2}}V_0+b(t_0,V_0) = 0\quad\mbox{for } \xi\in
     \mathbb{R},
\quad\;
V_0(-\infty)=\varphi_1(t_0), \;\; V_0(\infty) = \varphi_2(t_0).
\end{equation}
We shall see shortly that (\ref{V_0_nat}) has a solution $V_0(\xi)$,
but this solution is not unique as $V_0(\xi\pm C)$ is also a
solution for any constant $C$. Once we know that $V_0$ exists and
is a strictly increasing function,
consider a specific solution $\hat V_0$ of (\ref{V_0_nat}) subject to
the parametrization
\begin{equation}\label{V0_hat}
\hat V_0(0)=\varphi_0(t_0).
\end{equation}
One might expect $u(x)=\varphi_0(t_0)$ to hold at $x=t_0$ and
thus the interior layer to be described by $\hat V_0(\xi)$.
It is not the case, however; as we shall see below,
$u(x)=\varphi(t_0)$ at $x=t=t_0+\eps t_1+\eps^2 t_2+\cdots$, and
the interior layer is described by $\hat V_0(\xi- t_1-\eps t_2-\cdots)$.
Here $t_1$, $t_2, \ldots$ are independent of $\eps$ and can be found
when constructing an asymptotic expansion of $u$,
in particular, the values of $t_1$ and $t_2$ are specified in the proof of
Lemma~\ref{lem_u_as}.
In our analysis, we take $t=t_0+\eps t_1+\eps^2 t_2$, skipping higher-order terms, and
 invoke a perturbed version of $\hat V_0(\xi- t_1-\eps t_2)$
defined by
\begin{equation}\label{V0_hat_V0}
V_0(\xi;p)=\hat V_0(\xi-\bar t_1+p),
\qquad
\bar t_1=t_1+\eps t_2.
\end{equation}
\end{subequations}
Here the parameter $p$ satisfies $|p|\le p^*$ for any 
fixed positive constant $p^*$,
but will typically take very small values.

\begin{lemma}\label{lem:V0}
Set $\bar\gamma^2: = \min_{k=1,2}b_u(t_0,\varphi_k(t_0))>\gamma^2$,
where $\gamma>0$ is from $(A3)$.
For any constant $\bar t_1$ and all $|p|\le p^*$, there exist unique monotone solutions
$\hat V_0(\xi)$ and $V_0(\xi;p)$ of (\ref{V_0_nat}) that satisfy (\ref{V0_hat_V0}),\,(\ref{V0_hat}).
Furthermore, $\hat V_0$ and $V_0$ are in $C^\infty(\mathbb{R})$,
\begin{equation}\label{chi_def}
\hat\chi(\xi):={\textstyle\frac{d}{d \xi}}\hat V_0(\xi)>0,\qquad
\chi(\xi;p):={\textstyle\frac{d}{d \xi}}V_0(\xi;p)>0
\qquad\mbox{for }\xi\in\mathbb{R}.
\end{equation}
For any arbitrarily small but fixed $\lambda\in(0,\bar\gamma)$,
there is a constant $C_\lambda$ such that
%
\begin{equation}\label{chi_decay}
\hat \chi(\xi)+\chi(\xi;p)
\le C_\lambda  e^{-(\bar\gamma-\lambda)|\xi|}
\qquad\mbox{for}\;\;\xi\in\mathbb{R},\;\;|p|\le p^*.
\end{equation}
%
There are constants $C'$ and $C''$ such that for all $|p|\le p^*$
one has
\begin{equation}\label{V0_chi}
C'\chi\le V_0-\varphi_1(t_0)\le C''\chi\;\;\mbox{for}\;\xi<0,
\quad
C'\chi\le \varphi_2(t_0)-V_0\le C''\chi\;\;\mbox{for}\;\xi>0.
\end{equation}
\end{lemma}

\begin{proof}
In view of $(A1)$--$(A5)$, these properties follow from the proof of \cite[Lemma 2.1]{Fife}
or a slight extension of the proof of \cite[Lemma~2.1]{KoSt04}
using phase-plane analysis. 
\end{proof}

An interior-layer solution $u$ of problem (\ref{1.1}) can be regarded as having a boundary layer at $t_0$ on each of
the subintervals $[0,t_0]$ and $[t_0,1]$.
Therefore we shall construct standard  second-order boundary-layer asymptotic expansions
for $u$
on each of these sub-intervals. As $u(t_0)$ is unknown,
we impose the condition $u_{\rm as}(t_0)=V_0(0;p)$,
or equivalently $u_{\rm as}(t_0)=\hat V_0(-t_1-\eps t_2+p)$.
Here
$t_1$ and $t_2$ will be chosen in Lemma~\ref{lem_u_as}
to match the two asymptotic expansions at $x=t_0$,
while varying the parameter $p$
will be used in the construction of sub- and super- solutions.
We partly follow the asymptotic analyses of
\cite[Section 2.3.2]{VB2} and \cite[Section 3]{Nef}, where
the location of the interior layer was outlined in the form  of an
asymptotic expansion $ t=t_0 + \eps t_1 + \eps^2 t_2+\cdots$.
 Our asymptotic analysis differs from
these earlier works in that  we
expand about the point $t_0$ instead of about the point $t$ (which
is a priori unknown); this is useful in the subsequent numerical
analysis because our layer-adapted mesh will be centred on the known
point $t_0$.

As $u_{\rm as}(t_0)=V_0(0;p)$, the resulting asymptotic expansion $u_{\rm as}(x)=u_{\rm as}(x;p)$
will also involve the parameter $p$ as follows:
\begin{equation}\label{u_as}
u_{\rm as}(x;p):=u_0(x)+\eps^2 u_2(x)+v_0(\xi;p)+\eps v_1(\xi;p)
+\eps^2 v_2(\xi;p).
\end{equation}
Here the smooth component $u_0+\eps^2 u_2$ is defined by
\begin{equation}\label{u02}
u_0(x):=\left\{\begin{array}{ll}
\varphi_1(x),&x\in[0,t_0),\\
\varphi_2(x),&x\in(t_0,1],
\end{array}\right.,
\qquad
u_2(x):=u_0''/b_u(x,u_0),
\end{equation}
so that $F[u_0+\eps^2 u_2]=O(\eps^4)$ for $x\in[0,1]\setminus\{t_0\}$.

To describe the layer component $v_0+\eps v_1+\eps^2 v_2$,
we make use of the auxiliary function
\begin{equation}\label{B_def}
B(x,s):=b(x, u_0(x)+s),
\end{equation}
which is clearly discontinuous at $x=t_0$.
By (A2), we have ${\textstyle\frac{\partial^m}{\partial x^m}} B(x,0)=0$
for all $x\neq t_0$,
which implies that
\begin{equation}\label{B_xxx}
|{\textstyle\frac{\partial^m}{\partial x^m}} B(x,s)|\le C |s|
\qquad\mbox{for}\;\; x\in[0,1]\setminus\{t_0\},\;\; s\in\mathbb{R},\;\;
m=0,1,2 .
\end{equation}
Furthermore, we use the notation
\begin{equation}\label{t_0hat}
\hat t_0=\hat t_0(x):=\left\{\begin{array}{ll}
t_0^-,&x\in[0,t_0)\\
t_0^+&x\in(t_0,1]
\end{array}\right. ;
\end{equation}
thus, e.g., $u_0(\hat t_0)=\varphi_1(t_0)$
and $B(\hat t_0,s)=b(t_0,\varphi_1(t_0)+s)$
for $x<t_0$, while
$u_0(\hat t_0)=\varphi_2(t_0)$
and $B(\hat t_0,s)=b(t_0,\varphi_2(t_0)+s)$
for $x>t_0$.

The zero-order boundary-layer function $v_0(\xi)=v_0(\xi;p)$ is defined by
\begin{equation}\label{v_0}
-{\textstyle\frac{d^2}{d \xi^2}}v_0+B(\hat t_0,v_0)=0,
\quad v_0(0^\pm)=V_0(0;p)-u_0(t_0^\pm),
\quad v_0(\pm\infty)=0.
\end{equation}
Comparing this with (\ref{V_0_nat}), we see that
\begin{equation}\label{v0_V0}
v_0(\xi;p)= V_0(\xi;p)-u_0(\hat t_0).
\end{equation}

Higher-order boundary-layer components $v_1(\xi)=v_1(\xi;p)$
and $v_2(\xi)=v_2(\xi;p)$ and defined by
\begin{equation}\label{v_1}
[-{\textstyle\frac{d^2}{d \xi^2}}+ B_s(\hat t_0,v_0)]\,v_1=-\xi B_x(\hat t_0,v_0),
\quad
v_1(0)=v_1(\pm\infty)=0,
\end{equation}
\begin{equation}\label{v_2}
\begin{array}{c}
[-{\textstyle\frac{d^2}{d \xi^2}}+ B_s(\hat t_0,v_0)]\,v_2=\psi_2(\xi),
\quad
v_2(0^\pm)=-u_2(t_0^\pm),\quad v_2(\pm\infty)=0,\\
\psi_2(\xi):=-{\textstyle\frac{\xi^2}2} B_{xx}(\hat t_0,v_0)-\xi v_1 B_{xs}(\hat t_0,v_0)
-{\textstyle\frac{v_1^2}2} B_{ss}(\hat t_0,v_0)\\
-u_2(\hat t_0)[B_s(\hat t_0,v_0)-B_s(\hat t_0,0)].
\qquad\quad
\end{array}
\end{equation}
The functions $v_1$ and $v_2$ depend on $p$ since they are defined using $v_0(\xi;p)$.

Note that $v_0$ and $v_2$ have a discontinuity at $\xi=0$, but
 $u_0+v_0=[u_0(x)-u_0(\hat t_0)]+V_0$
and $u_2+v_2$
 are continuous at $x=t_0$.
Thus $u_{\rm as}(x;p)$ is continuous for $x\in[0,1]$.

Given any suitable function $v(x)$, introduce the functional
$$
\Phi[v(\cdot)]:=\eps {\textstyle\frac{dv}{d x}}\bigr|_{x=t_0^-}
-\eps{\textstyle\frac{d v}{d x}}\bigr|_{x=t_0^+}\,,
$$
which will be used to match our asymptotic expansion at $x=t_0$;
see Lemma~\ref{lem_u_as}.

To establish the existence and properties of  $v_1$ and $v_2$,
note that (\ref{v_1}) and (\ref{v_2}) are particular cases of a
 general problem
\begin{equation}\label{nu}
[-{\textstyle\frac{d^2}{d \xi^2}}+ B_s(\hat t_0, v_0)]\,\nu=\psi(\xi)\quad
\mbox{for }\xi\in\mathbb{R}\setminus\{0\},
\quad\nu(0^\pm)=\nu^\pm_0,\quad\nu(\pm\infty)=0,
\end{equation}
for which we have the following result.

\begin{lemma}\label{lem_nu_psi}
Let $|\psi(\xi)|\le C(1+|\xi|^k)\,\chi(\xi)$. Then there exists a solution $\nu$ of
 problem~(\ref{nu}), which satisfies $|\nu(\xi)|\le C(1+|\xi|^{k+1})\,\chi(\xi)$ and
\begin{equation}\label{Phi_nu}
\Phi[\nu]=
\nu'(0^-)-\nu'(0^+)={\textstyle\frac{1}{\chi(0)}}\Bigl(
-\int_{-\infty}^\infty\!\! \psi(\xi)\,\chi(\xi) \,d\xi
+[\nu_0^--\nu^+_0]\,\chi'(0)
\Bigr).
\end{equation}
Furthermore, if $\psi(\xi)\ge0$ and $\nu^\pm\ge 0$, then $\nu(\xi)\ge 0$
for all $\xi$.
\end{lemma}
\begin{proof}
The desired assertions follow from the explicit solution formula
$$
    \nu(\xi)= \chi(\xi)\int^{0}_\xi\! \chi^{-2}(\eta)
    \int^\eta_{-\infty} \!\!\chi(t)\,\psi(t)\,dt\,d\eta
    + {\textstyle\frac{\nu_0^-}{\chi(0)}}\,\chi(\xi)\qquad\mbox{for}\;\xi<0,
$$\vspace{-0.5cm}$$
 \nu(\xi)=
    \chi(\xi)\int_{0}^\xi\! \chi^{-2}(\eta) \int_\eta^\infty
    \!\!\chi(t)\,\psi(t)\,dt\,d\eta
    + {\textstyle\frac{\nu_0^+}{\chi(0)}}\,\chi(\xi)\qquad\mbox{for}\;\xi>0,
$$
which is obtained by variation of parameters
noting that, by (\ref{V_0_nat}),\,(\ref{chi_def}), the function $\chi$ satisfies
$-{\textstyle\frac{d^2}{d \xi^2}}\chi+\chi B_s(\hat t_0, v_0)=0$;
see \cite[Lemma~2.2]{Fife}.
Now, a calculation yields (\ref{Phi_nu}) and the other assertions.

We shall show, e.g., that $|\nu|\le C(1+|\xi|^{k+1})\,\chi$
for $\xi<0$. As it follows from (\ref{V0_chi}),\,(\ref{v0_V0}) that
$C'\chi\le v_0\le C''\chi$, then we have
$\chi^2 dt\le \frac1{C'}v_0\,dv_0=\frac1{2C'}d(v_0^2)$
and therefore $|\psi|\chi\,dt\le C(1+|t|^k)\,d(v_0^2)$.
Thus, integrating by parts $k$ times, one gets
$|\int^\eta_{-\infty} \!\chi(t)\psi(t)\,dt| \le C\chi^2(1+|\eta|^k)$.
The desired bound follows.
%
\end{proof}

We shall now apply Lemma~\ref{lem_nu_psi} to establish properties of
$v_0$, $v_1$ and $v_2$.

\begin{lemma}\label{lem:vj}
For any constant $t_1$ and $t_2$ in (\ref{V0_hat_V0}),
there exist solutions $v_0$, $v_1$ and $v_2$
of problems (\ref{v_0}),\, (\ref{v_1}) and (\ref{v_2}),
respectively.
The function $v_0$ satisfies
\begin{equation}\label{v0_sign}
({\rm sgn}\,\xi)\cdot v_0(\xi)> 0
\qquad\mbox{and}\qquad
|v_0(\xi)|\le C''\chi(\xi) \qquad
\text{for }\xi\in\mathbb{R}\setminus\{0\}.
\end{equation}
Furthermore, assuming that $|t_1|+|t_2|\le C$ and $|p|\le p^*$,
for any arbitrarily small but fixed $\lambda\in(0,\bar\gamma)$,
there is a constant $C_\lambda$ such that
%
\begin{equation}\label{v012_der}
\bigl|{\textstyle{\frac{d^k }{d \xi^k}}}  v_j \bigr|
\le C_\lambda  e^{-(\bar\gamma-\lambda)|\xi|}
\qquad\mbox{for}\;\;\xi\in\mathbb{R}\setminus\{0\},\;\; j=0,1,2,\;\; k=0,\ldots,6.
\end{equation}
\end{lemma}
\begin{proof}
The existence and properties (\ref{v0_sign}) of the function $v_0$ as well
the bound (\ref{v012_der}) for $j=0$, $k=0,1$
follow from the observation (\ref{v0_V0}) combined with
(\ref{t_0hat}) and Lemma~\ref{lem:V0}.
Next, the existence of
$v_1$ and $v_2$ and the bound (\ref{v012_der}) for $j=1,2$, $k=0$
are obtained by applying Lemma~\ref{lem_nu_psi} to problems (\ref{v_1}) and (\ref{v_2}),
in which the right-hand sides are estimated using (\ref{B_xxx}) with $m=1,2$
and also $|v_0|\le C''\chi$.
Similarly, the higher-order derivatives of $v_0$, $v_1$ and $v_2$ all satisfy problems of type
(\ref{nu}) with various data, so the bound (\ref{v012_der}) for them is obtained
by again applying Lemma~\ref{lem_nu_psi}.
\end{proof}

The main result of this section is as follows.

\begin{lemma}\label{lem_u_as}
For the asymptotic expansion $u_{\rm as}(x;p)$ from (\ref{u_as}) we have
\begin{subequations}
\begin{equation}\label{Fuas}
Fu_{\rm as}(x;p)=O(\eps^3)\qquad\mbox{for}\; x\in(0,1)\setminus\{t_0\}.
\end{equation}
Furthermore, there exist values of $t_1$ and $t_2$ in (\ref{V0_hat_V0}),
independent of $\eps$ and $p$, and
positive constants $C_1$, $C_2$ and $\eps^*=\eps^*(p^*)$
such that for all $\eps\le\eps^*$ and $0<|p|\le p^*$ we have
\begin{equation}\label{Phi_p}
({\rm sgn}p)\cdot\Phi[u_{\rm as}(\cdot;p)]\ge C_1 \eps |p|-C_2\eps^3.
\end{equation}
\end{subequations}
\end{lemma}

\begin{proof}
The relation (\ref{Fuas}) is a standard outcome of the method of asymptotic
expansions that was applied to generate the terms in (\ref{u_as}).

To establish (\ref{Phi_p}), note that
(\ref{u02}) implies $\Phi[u_0]=\eps[\varphi_1'(t_0)-\varphi_2'(t_0)]$.
As we also have $\Phi[v_0]=\Phi[V_0]=0$ and $\Phi[\eps^2 u_2]=O(\eps^3)$, then
\begin{equation}\label{Phi_1}
\Phi[u_{\rm as}]=\eps[\varphi_1'(t_0)-\varphi_2'(t_0)]+
\eps\Phi[ v_1]+\eps^2 \Phi[v_2]+O(\eps^3).
\end{equation}
By applying (\ref{Phi_nu}) to problem (\ref{v_1}), we get
$$
\Phi[ v_1]
={\textstyle\frac{1}{\chi(0)}}\int_{-\infty}^\infty \xi\,
B_x(\hat t_0,v_0)\,\chi \,d\xi.
$$
Note that
$B_x(x,v_0)=b_x(x,u_0(x)+v_0)+u_0'(x)\,b_u(x,u_0(x)+v_0)$ for $x\neq t_0$.
Combining this with $u_0(\hat t_0)+v_0=V_0$ yields
$B_x(\hat t_0,v_0)=b_x(t_0,V_0)+u_0'(\hat t_0)\,b_u(t_0,V_0)$
and thus
$$
\Phi[ v_1]={\textstyle\frac{1}{\chi(0)}}\!
\int_{-\infty}^\infty \!\!\!\xi\, b_x(t_0,V_0)\,\chi \,d\xi
+{\textstyle\frac{\varphi_1'(t_0)}{\chi(0)}}\!\int_{-\infty}^0 \!\!\!\xi\, b_u(t_0,V_0)\,\chi \,d\xi
+{\textstyle\frac{\varphi_2'(t_0)}{\chi(0)}}\!\int_{0}^\infty\!\!\! \xi\, b_u(t_0,V_0)\,\chi \,d\xi.
$$
For the second term on the right-hand side, we note
that $b_u(t_0,V_0)\chi={\textstyle\frac{d}{d\xi}}[b(t_0,V_0(\xi))] $,
so an integration by parts yields
$$
\int_{-\infty}^0\!\! \xi\, b_u(t_0,V_0)\,\chi \,d\xi
=-\int_{-\infty}^0\!\! b(t_0,V_0) \,d\xi
=-\int_{-\infty}^0\!\! \chi'(\xi) \,d\xi=-\chi(0).
$$
Combining this with a similar estimate for the third term, we arrive at
\begin{equation}\label{Phi_v1a}
\Phi[ v_1]={\textstyle\frac{1}{\chi(0)}}\!
\int_{-\infty}^\infty \!\!\xi\, b_x(t_0,V_0)\,\chi \,d\xi
-[\varphi_1'(t_0)-\varphi_2'(t_0)].
\end{equation}
For the first term here,
recall (\ref{V0_hat_V0}) and therefore
switch to the variable $\hat\xi=\xi-\bar t_1+p$
so that the resulting integral involves the functions $\hat V_0$
and $\hat\chi$,
which are independent of $\bar t_1$ and $p$, rather than
$V_0(\xi;p)$ and $\chi(\xi;p)$:
\begin{equation}\label{Phi_v1b}
\int_{-\infty}^\infty \xi \,b_x(t_0,V_0)\,\chi \,d\xi
=\int_{-\infty}^\infty \!(\hat\xi+\bar t_1-p)
\,b_x(t_0,\hat V_0(\hat \xi))\,\hat \chi(\hat \xi) \,d\hat\xi
=C_{I\!I}-(\bar t_1-p)C_I.
\end{equation}
Here
$$
C_{I\!I}=\int_{-\infty}^\infty \!\hat\xi
\,b_x(t_0,\hat V_0(\hat \xi))\,\hat \chi(\hat \xi) \,d\hat\xi
$$
is a fixed constant,  independent of $\bar t_1$ and $p$,
and
$$
C_I=-\int_{-\infty}^\infty \!
b_x(t_0,\hat V_0(\hat \xi))\,\hat \chi(\hat \xi) \,d\hat\xi
=-\int_{\varphi_1(t_0)}^{\varphi_2(t_0)}  b_x(t_0,v) \,dv>0
$$
is a positive constant that appears in Assumption (A5).
We now choose $t_1:=C_{I\!I}/C_I$ so that
$\bar t_1=C_{I\!I}/C_I+\eps t_2$ and therefore
$C_{I\!I}-(\bar t_1-p)C_I=(p-\eps t_2)C_I$.
Combining this with (\ref{Phi_v1a}) and (\ref{Phi_v1b}) yields
$$
\Phi[ v_1]={\textstyle\frac{1}{\chi(0)}}\!
(p-\eps t_2)C_I
-[\varphi_1'(t_0)-\varphi_2'(t_0)].
$$
Substituting this result in (\ref{Phi_1}), we arrive at
$$
\Phi[u_{\rm as}]=
\eps {\textstyle\frac{1}{\chi(0)}}\!
(p-\eps t_2)C_I
+\eps^2 \Phi[v_2]+O(\eps^3).
$$
Now, by applying (\ref{Phi_nu}) to problem (\ref{v_2}), we get
 $\Phi[v_2]=\frac{1}{\chi(0)}[C_{I\!I\!I}+O(p+\eps |t_2|)]$,
 where $C_{I\!I\!I}$ equals the expression in the parentheses of formula (\ref{Phi_nu})
 evaluated using the data of  problem (\ref{v_2}) in the case of $p=0$
 and $\bar t_1=t_1$; thus $C_{I\!I\!I}$ is independent of $p$ and $\eps$.
Now choosing $t_2:=C_{I\!I\!I}/C_I$ yields
\begin{equation}\label{Phi_2}
\Phi[u_{\rm as}]
={\textstyle\frac{1}{\chi(0)}}\eps C_Ip+O(\eps^2 p)+O(\eps^3).
\end{equation}
Note that there exists $C'$ such that
$\chi(0)=\hat\chi(-p+t_1+\eps t_2)$ satisfies $\frac1{\chi(0)}\ge C'$ for all
$\eps\le 1$ and $|p|\le p^*$. Thus choose $C_1:=\frac12 C_I C'$
so that ${\textstyle\frac{1}{\chi(0)}}\eps C_I\ge 2\eps C_1$.
Finally, by choosing
 $\eps^*$ sufficiently small and $C_2$ sufficiently large
 so that  the $O$ terms in (\ref{Phi_2}) satisfy
 $|O(\eps^2 p)|\le \eps C_1|p|$ and $|O(\eps^3)|\le C_2\eps^3$,
 we establish (\ref{Phi_p}).
\end{proof}

Note that the numerical solution of problem~\eqref{1.1}
presents substantial difficulties and instabilities \cite{KoStMain}.
In that paper we describe a special numerical treatment
for particularly small values of $\eps$ that is
based on the following result:
\begin{lemma}
Let $ \tau =
{\textstyle\frac {C_{\tau}}{\bar\gamma}}\,\eps \ln N$
for some $C_\tau>2$ and $N\ge 2$.
Then
the asymptotic expansion $u_{\rm as}(x;0)$  of (\ref{u_as})
can be written as
$$
u_{\rm as}(x;0)={\mathcal U}(x,\eps)+O(\eps\ln N+N^{-2}),
\quad
{\mathcal U}(x,\eps):=
\left\{\begin{array}{cl}
V_0(\frac{x-t_0}{\eps};0),& |x-t_0|\le\tau,\\
u_0(x), & |x-t_0|>\tau.
\end{array}
\right.
$$
\end{lemma}
\begin{proof}
Note that (\ref{u_as}) immediately  implies that
$u_{\rm as}(x;0)=u_0(x)+v_0(\xi;0)+O(\eps)$.

\noindent
(i) Let $|x-t_0|\le \tau$. By (\ref{v0_V0}),
$u_{\rm as}(x;0)=V_0(\xi;0)+[u_0(x)-u_0(\hat t_0)]+O(\eps)$.
Here, by virtue of (\ref{u02}) and (\ref{t_0hat}),
one has $|u_0(x)-u_0(\hat t_0)|\le C|x-t_0|\le C\tau$.
Consequently
$u_{\rm as}(x;0)=V_0(\xi;0)+O(\tau+\eps)$,
which yields the desired result.

\noindent (ii) Now let  $|x-t_0|>\tau$, i.e.,  $|\xi|>\tau/\eps$.
Note that (\ref{v0_sign}) combined with (\ref{chi_decay}) yields
$|v_0|\le  C''  C_\lambda\,
e^{-(\bar\gamma-\lambda)|\xi|}$.
Choosing $\lambda$ sufficiently small so that
$C_\tau(1-\lambda/\bar\gamma)\ge 2$, one gets
$e^{-(\bar\gamma-\lambda)\tau/\eps}\le N^{-2}$. 
Consequently
$u_{\rm as}(x;0)=u_0(x)+O(N^{-2}+\eps)$ and the desired result follows.
\end{proof}

\section{Perturbed asymptotic expansion, sub- and super-solutions}~\label{sec_per_as_exp}
For the numerical analysis that appears in \cite{KoStMain}
we now perturb the asymptotic expansion $u_{\rm as}(x;p)$ of (\ref{u_as})
as follows: set
\begin{equation}\label{beta_def}
\beta(x)=
\beta(x;p,p',\hat h):=u_{\rm as}(x;p)
+ p'\, [v_*(\xi;p)+C_0]+\hat h^2\, z(\xi;p).
\end{equation}
Clearly $\beta$ is a small perturbation of $u_{\rm as}$ when
the parameters $p'$ and $\hat h$ are small.
In this definition,
the parameter $\hat h$ is related to the mesh used in~\cite{KoStMain}
as the component $\hat h^2 \,z(\xi;p)$ is added to compensate for the principal
part of the truncation error produced when the finite difference operators of~\cite{KoStMain}
are applied to $u_{\rm as}(x,t)$.
The component
 $p' [v_*(\xi;p)+C_0]$ is added
to ensure that $({\rm sgn}p')\cdot F(u_{\rm as}+ p' [v_*+C_0])\ge 0$.

The functions $v_*(\xi)=v_*(\xi;p)$ and $z(\xi)=z(\xi;p)$ used in (\ref{beta_def})
are defined by
\begin{equation}\label{v_star}
[-{\textstyle\frac{d^2}{d \xi^2}}+ B_s(\hat t_0,v_0)]\,v_*=
|v_0|
,\quad\mbox{for}\;\xi\in\mathbb{R}\setminus\{0\},
\qquad
v_*(0)=v_*(\pm\infty)=0,
\end{equation}
\begin{equation}\label{z_prob}
[-{\textstyle\frac{d^2}{d \xi^2}}+ B_s(\hat t_0,v_0)]\,z=
{\textstyle\frac1{12}\frac{d^4}{d\xi^4}}V_0
\quad\mbox{for}\;\xi\in\mathbb{R},
\qquad
z(0)=z(\pm\infty)=0.
\end{equation}
The functions $v_*$ and $z$ depend on $p$ since they are defined using $v_0(\xi;p)$
and $V_0(\xi;p)$.

\begin{lemma}
Assume that $|p|\le p^*$ for some positive constant $p^*$. Then
there exist solutions $v_*$ and $z$ of problems (\ref{v_star}) and
(\ref{z_prob}) respectively, and for any arbitrarily small but fixed
$\lambda\in(0,\bar\gamma)$, there is a constant $C_\lambda$ such
that
%
\begin{equation}\label{v_star_z}
v_*\ge 0, \quad
\bigl|{\textstyle\frac{d^k   }{d \xi^k}}v_* \bigr|+
\bigl|{\textstyle\frac{d^k }{d \xi^k}} z  \bigr|
\le C_\lambda  e^{-(\bar\gamma-\lambda)|\xi|}
\quad\mbox{for}\;\;\xi\in\mathbb{R}\setminus\{0\},\;\;k=0,\ldots,4.
\end{equation}
Furthermore, 
there exist
positive constants $C_1$, $C_2$, $C_3$ and $\eps^*=\eps^*(p^*)$
such that for all $\eps\le\eps^*$ and $0<|p|\le p^*$ we have
\begin{equation}\label{Phi_beta}
({\rm sgn}p)\cdot \Phi[\beta(x;p,p',\hat h)]\ge
C_1 \eps |p|-C_2\eps^3-C_3 |p'|.
\end{equation}
\end{lemma}

\begin{proof}
The existence and properties (\ref{v_star_z}) of $v_*$ and $z$ are obtained by
applying Lemma~\ref{lem_nu_psi}
to problems (\ref{v_star}) and (\ref{z_prob}), respectively.
Furthermore,  the relation (\ref{Phi_nu}) from this lemma
yields the values of $\Phi[v_*]$ and $\Phi[z]$.
For $\Phi[v_*]$, using the sign property of $v_0$ from (\ref{v0_sign}),
we get
\begin{eqnarray}\label{Phi_v_star}
\Phi[ v_*]&=&-{\textstyle\frac{1}{\chi(0)}}\int_{-\infty}^\infty \!|v_0|\chi \,d\xi
=-{\textstyle\frac{1}{\chi(0)}}\bigl(\int_{-\infty}^0 \!v_0\chi \,d\xi
-\int_{0}^\infty\! v_0\chi \,d\xi\bigr)
\\\nonumber
&=&-{\textstyle\frac{1}{\chi(0)}}\bigl([v_0(0^-)]^2+[v_0(0^+)]^2\bigr)\ge -C_3
\qquad\mbox{for all~~} |p|\le p^*.
\end{eqnarray}
Here, in view of (\ref{chi_def}),
we used $v_0\chi=\frac12 (v_0^2)'$ and $v_0(\pm\infty)=0$,
and it was understood that $\chi(0)=\chi(0;p)$ and $v_0(0^\pm)=v_0(0^\pm;p)$.
Next, for $\Phi[z]$,
noting that $\frac{d^4}{d\xi^4}V_0=\chi'''$
and $\chi'=V_0''=b(t_0,V_0)$,   we get
$$
\Phi[z]
=-{\textstyle\frac1{12\chi(0)}}\int_{-\infty}^\infty
\chi'''\chi \,d\xi
={\textstyle\frac1{12\chi(0)}}\int_{-\infty}^\infty
\chi''\chi' \,d\xi
={\textstyle\frac1{24\chi(0)}}[\chi'(\xi)]\Bigr|_{-\infty}^\infty=0.
$$
Thus
$\Phi[\beta]=\Phi[u_{\rm as}]+ p'\Phi[v_*]$, and combining this with
(\ref{Phi_v_star}) and (\ref{Phi_p})
yields (\ref{Phi_beta}).
\end{proof}

\begin{lemma}
There exist  positive constants $C_0$, $C_4$,
$p'^*$ and $\eps^*$ such that for all
$x\in(0,1)\setminus\{t_0\}$, $\eps\le\eps^*$,  $|p|\le p^*$, $0<|p'|\le p'^*$,
the function $\beta$ of (\ref{beta_def}) satisfies
\begin{equation}\label{Fbeta_C0}
({\rm sgn}p')\cdot[F\beta-{\textstyle\frac{\hat h^2}{12}\frac{d^4}{d\xi^4}}V_0]
\ge {\textstyle\frac12}C_0  |p'| \gamma^2
-C_4(\eps^3+\eps \hat h^2+  \hat h^4).
\end{equation}
\end{lemma}

\begin{proof}
We partly imitate the analysis of ~\cite[Lemma~3.2]{KoSt04}.
As, by (\ref{Fuas}), we have
$Fu_{\rm as}=O(\eps^3)$, it suffices to estimate $F\bigr|^{\beta}_{u_{\rm as}}$,
where we use the notation $F\bigr|_{v}^{w}:=Fw-Fv$
 for any two functions $v$ and $w$.
Noting that for $u_{\rm as}$ of (\ref{u_as}) we have $u_{\rm as}=u_0+v_0+O(\eps)$,
which, by (\ref{B_def}), implies
$b_u(x,u_{\rm as})=B_s(x,v_0)+O(\eps)$, we obtain
$$
F\bigr|_{u_{\rm as}}
^{u_{\rm as}+ p' v_*+\hat h^2 z}
=- p' {\textstyle\frac{d^2}{d \xi^2}}v_*-\hat h^2 {\textstyle\frac{d^2}{d \xi^2}}z
+ (p' v_*+\hat h^2 z)[B_s(x,v_0)
+O( \eps+p'+\hat h^2)].
$$
Next, using (\ref{v_star}) and (\ref{z_prob})
for ${\textstyle\frac{d^2}{d \xi^2}}v_*$ and ${\textstyle\frac{d^2}{d \xi^2}}z$
and then noting that for $x\neq t_0$ we have
$|B_s(x,v_0)-B_s(\hat t_0,v_0)|\le C|x-t_0||v_0|\le C\eps$,
yields
\begin{equation}\label{Fuas_a}
F\bigr|_{u_{\rm as}}
^{u_{\rm as}+ p' v_*+\hat h^2 z}=
p'|v_0|+{\textstyle\frac{\hat h^2}{12}\frac{d^4}{d\xi^4}}V_0
+[p'+\hat h^2]O( \eps+p'+\hat h^2).
\end{equation}
Similarly, as $b_u(x,u_{\rm as}+ p' v_*+\hat h^2 z)=B_s(x,v_0)+O(\eps +p'+\hat h^2)$
and $B_s(x,v_0)=B_s(x,0)-\lambda (x)|v_0|$,
where $|\lambda (x)|\le C_5$ for some $C_5$, we get
$$
F\bigr|^{\beta}
_{u_{\rm as}+ p' v_*+\hat h^2 z}
=F\bigr|^{u_{\rm as}+ p' v_*+\hat h^2 z+C_0p'}
_{u_{\rm as}+ p' v_*+\hat h^2 z}
=C_0 p'[B_s(x,0)-\lambda (x)|v_0|
+O( \eps +p'+\hat h^2)].
$$
Combining this with (\ref{Fuas_a}) and $Fu_{\rm as}=O(\eps^3)$, yields
$$
F\beta-{\textstyle\frac{\hat h^2}{12}\frac{d^4}{d\xi^4}}V_0=
C_0 p'B_s(x,0)+p'|v_0|(1-C_0 \lambda (x))+O(\eps p'+ p'^2)
+O(\eps^3+\eps \hat h^2+ \hat h^4).
$$
Note that here $B_s(x,0)=b_u(x,u_0(x))>\gamma^2$ for $x\neq t_0$, by Assumption (A3), so
$C_0 |p'|B_s(x,0)\ge C_0|p'|\gamma^2$.
Now, choosing $C_0=C_5^{-1}$ so that $(1-C_0 \lambda (x))\ge 0$,
and also $p'^*$ and $\eps^*$ sufficiently small so that
$|O(\eps p'+ p'^2)|\le \frac12 C_0|p'|\gamma^2$,
we get the desired assertion (\ref{Fbeta_C0}) for some constant $C_4$.
\end{proof}

\begin{lemma}\label{lem_beta_monot}
Let $p\ge 0$, $p'=C'\eps p$ for some positive constant $C'$,
and
$\hat h^2\le C\eps^\mu$ for some fixed $\mu\in(0,1]$. Then
there exists $\eps^*=\eps^*(C',\mu)$ such that for the function
$\beta$ from (\ref{beta_def}) we have
\begin{equation}\label{beta_monot}
\beta(x;-p,-p',\hat h)\le \beta(x;p,p',\hat h)
\qquad\mbox{for}\;\;x\in[0,1],\;\; \eps\le \eps^*,\;\; |p|\le p^*.
\end{equation}
Furthermore, for any arbitrarily small but fixed $\lambda\in(0,\bar\gamma)$,
there is a constant $C_\lambda$ such that $u_{\rm as}$ from (\ref{u_as}) satisfies
\begin{equation}\label{beta_u_as}
|\beta(x;\pm p,\pm p',\hat h)-u_{\rm as}(x;0)|\le
C_\lambda (|p|+\hat h^2 )e^{-(\bar\gamma-\lambda)|\xi|}
+ C \eps|p|.
\end{equation}
\end{lemma}

\begin{proof}
Fix $\hat h$, and consider
$\tilde\beta(x;p):=\beta(x;p,C'\eps p,\hat h)$. As $\beta$ is continuous on $[0,1]$,
to establish (\ref{beta_monot}),
it suffices to show that $ \frac{\partial }{\partial p}\tilde\beta\ge 0$ for $x\neq t_0$.
First, we note that
$$
{\textstyle\frac{\partial}{\partial p}}v_0
=\chi,
\quad
|{\textstyle\frac{\partial}{\partial p}} v_j|\le C(1+\xi^{2j})\chi
\;\;\mbox{for~}j=1,2,
\quad
|{\textstyle\frac{\partial}{\partial p}} v_*|+|{\textstyle\frac{\partial}{\partial p}}z |
\le C(1+|\xi|)\chi.
$$
The relation for $v_0$ here follows from (\ref{v0_V0}) and (\ref{chi_def}).
The other relations are obtained
by differentiating problems (\ref{v_1}),\,(\ref{v_2}),\,(\ref{v_star}) and (\ref{z_prob})
with respect to $p$,
which yields four problems of the type
(\ref{nu}) for $\frac{\partial}{\partial p}v_{1,2}$,
$\frac{\partial}{\partial p}v_*$ and $\frac{\partial}{\partial p}z$, respectively.
By Lemma~\ref{lem_nu_psi} applied to this problems, the above estimates follow.
Now a calculation, using (\ref{beta_def}), (\ref{u_as}) and $v_*\ge 0$, shows that
for some constant $C''$ we have
$$
{\textstyle\frac{\partial}{\partial p}}\tilde\beta
\ge \chi+C'\eps C_0- C''\eps^\mu (1+\xi^4)\chi.
$$
By (\ref{chi_decay}), there exists a sufficiently large constant $\bar C$ such that
for $|\xi|\ge \xi^*:=\bar C|\ln \eps|$ we have
$(1+\xi^4)\chi\le \frac{C'C_0}{C''}\eps^{1-\mu}$, which implies
$\frac{\partial }{\partial p}\tilde\beta\ge 0$ for $|\xi|\ge \xi^*$.
Otherwise, if $|\xi|\le \bar C|\ln\eps|$, then
$C''\eps^\mu (1+\xi^4)\le C \eps^\mu (1+|\ln\eps|^4)\le 1$,
provided that $\eps^*$ is sufficiently small; thus we again
get $\frac{\partial }{\partial p}\tilde\beta\ge 0$.
Thus we have proved (\ref{beta_monot}).

Similarly, one has $|\frac{\partial }{\partial p}\tilde\beta|\le
C(1+\xi^4)\chi+C_0\eps C'$.
Combining this with $u_{\rm as}(x;0)=\beta(x;0,0,\hat h)-\hat h^2 z
=\tilde\beta(x;0)-\hat h^2 z$
and exponential-decay estimates from  (\ref{chi_decay}) and (\ref{v_star_z}),
we get the remaining desired estimate (\ref{beta_u_as}).
\end{proof}

In this  section we established the properties (\ref{Phi_beta}),
(\ref{Fbeta_C0}) and (\ref{beta_monot}) of the function
$\beta=\beta(x;p,p',\hat h)$ that are used in~\cite{KoStMain} to
construct discrete sub- and super-solutions.


%
\bibliography{re_int1d}
\bibliographystyle{plain}

\end{document}